\documentclass[11pt,a4paper]{article}

\usepackage{amsmath}
\usepackage{amsthm}
\usepackage{amsfonts}

\parskip\smallskipamount
\newtheorem{theorem}{Theorem}
\theoremstyle{remark}
\newtheorem{rem}{Remark}

\renewcommand{\Pr}{\mathbf{P}}
\newcommand{\E}{\mathbf{E}}
\newcommand{\R}{\mathbb{R}}
\newcommand{\Rp}{\mathbb{R}_+}
\newcommand{\I}{\textbf{I}}

\parskip \smallskipamount

\title{A note on Veraverbeke's theorem}
\author{Stan Zachary}
\date{\today}

\begin{document}


 \begin{center}
   \Large\bfseries A note on Veraverbeke's theorem
 \end{center}

 \begin{center}
   Stan Zachary
   \footnotetext{
     {\it American Mathematical Society 1991 subject classifications.\/}
     Primary 60G70; secondary  60K30, 60K25


     }
 \end{center}

 \begin{center}
   \textit{Heriot-Watt University\\Edinburgh}
 \end{center}

\begin{quotation}\small
  We give an elementary probabilistic proof of Veraverbeke's Theorem
  for the asymptotic distribution of the maximum of a random walk with
  negative drift and heavy-tailed increments.  The proof gives
  insight into the principle that the maximum is in general attained
  through a single large jump. 
\end{quotation}

\section{Introduction}

Veraverbeke's Theorem (Veraverbeke (1977), Embrechts and Veraverbeke
(1982)) gives the asymptotic distribution of the maximum~$M$ of a
random walk with negative drift and heavy-tailed increments, and has
well-known applications to both queueing theory and risk theory.  A
simpler treatment of this result is given, for example, by Embrechts
\textit{et al.}\ (1997) or Asmussen (2000), and is based on renewal
theory and ascending ladder heights.  Nevertheless the underlying
intuition of the result is that the only significant way in which a
high value of the maximum can be attained is through ``one big jump''
by the random walk away from its mean path.  We give here a relatively
short proof from first principles which captures this intuition.  It
is similar in spirit to the existing probabilistic proof, but by
considering instead a first renewal time at which the random walk
exceeds a ``tilted'' level, the argument becomes more elementary.  In
particular subsequent renewals have an asymptotically negligible
probability under appropriate limits, and results from renewal
theory---notably the derivation and use of the Pollaczeck-Khinchine
formula---are not required.  We consider separately lower and upper
bounds, as a slightly weaker condition is required for the former.
Further, this treatment facilitates the derivation of bounds for
$\Pr(M>x)$ for any $x$---see Remark~\ref{rem:2} below.  The argument
leading to the lower bound is implicit in much of the literature (see,
for example, Asmussen \textit{et al.}\ (1999), Korshunov (2002),
Baccelli and Foss (2003, Theorem~4)).  We give it here for
completeness.  The argument leading to the upper bound is new.

\section{Veraverbeke's Theorem.}

We recall first some definitions and known properties.  For any
distribution function~$H$ on $\R$ let $\overline{H}(x)=1-H(x)$ for all
$x$.  A distribution function~$H$ on $\Rp$ is \emph{subexponential} if
and only if $\overline{H}(x)>0$ for all $x$ and
\begin{equation}\label{eq:1}
  \lim_{x\to\infty}\overline{H^{*n}}(x)/\overline{H}(x) = n,
  \qquad n\ge2,
\end{equation}
where $H^{*n}$ is the $n$-fold convolution of $H$ with itself.  (It is
sufficient to verify the condition~\eqref{eq:1} in the case~$n=2$.)
More generally, a distribution function~$H$ on $\R$ is subexponential
if and only if $H^+$ is subexponential, where $H^+=H\I_{\R_+}$ and
$\I_{\R_+}$ is the indicator function of $\R_+$.  In this case the
condition~\eqref{eq:1} continues to hold, but is no longer sufficient
for subexponentiality.  It is well known that if $H$ is
subexponential, then $H$ is \emph{long-tailed}, i.e.
\begin{equation}\label{eq:2}
  \lim_{x\to\infty}\frac{\overline{H}(x-h)}{\overline{H}(x)} = 1,
  \quad\text{for all fixed $h\in\R$}.
\end{equation}
(See, for example, Embrechts \textit{et al.}\ (1997, Lemma 1.3.5).  It
is of course sufficient to require that (\ref{eq:2}) hold for $h>0$;
the extension to $h<0$ follows by taking reciprocals.)  

For any distribution function~$H$ on $\R$ with finite mean define also
the integrated, or \emph{second-tail}, distribution function~$H^s$ by
\begin{equation}\label{eq:3}
  \overline{H^s}(x) = \min \bigl( 1,
    \int_x^{\infty} \overline{H}(t)\, dt
  \bigr).
\end{equation}
(When $H$ is supported on $\Rp$, the second-tail distribution is
usually, and in this case more naturally, defined instead by
$\overline{H^s}(x)=\mu^{-1}\int_x^{\infty}\overline{H}(t)\,dt$, where
$\mu=\int_0^{\infty}\overline{H}(t)\,dt$ is the mean of $H$.)

Now let $\{\xi_n\}_{n\ge1}$ be independent identically distributed
random variables with distribution function~$F$ such that
\begin{equation}\label{eq:4}
  \E\xi_1=-a<0.
\end{equation}
Let $S_0=0$, $S_n=\sum_{i=1}^n\xi_i$ for $n\ge1$.  Let
$M_n=\max_{0\le{}i\le{}n}S_i$ for $n\ge0$ and let $M=\sup_{n\ge0}S_n$.
Clearly $\Pr(M<\infty)=1$.  We are interested in the asymptotic
distribution of $\Pr(M>x)$ as $x\to\infty$.  Under an appropriate further
condition on the second-tail distribution function~$F^s$ this is
given by Theorem~1 below.  Part~(ii) of this result is Veraverbeke's
Theorem.  The proof we give is also well-adapted to the possibility of
obtaining bounds for $\Pr(M>x)$ for any $x$.

\begin{theorem}[Veraverbeke]~
  \begin{enumerate}
  \item[(i)] Suppose that, in addition to \eqref{eq:4}, $F^s$ is
    long-tailed.  Then
    \begin{equation}
      \label{eq:5}
      \liminf_{x\to\infty}\frac{\Pr(M>x)}{\overline{F^s}(x)} \ge \frac{1}{a}.
    \end{equation}
  \item[(ii)] Suppose that, in addition to \eqref{eq:4}, $F^s$ is
    subexponential.  Then
    \begin{equation}\label{eq:6}
      \lim_{x\to\infty}\frac{\Pr(M>x)}{\overline{F^s}(x)} = \frac{1}{a}.
  \end{equation}
  \end{enumerate}
\end{theorem}

\begin{proof}
  We derive first the lower bound given in~(i).  From \eqref{eq:4}, by
  the Weak Law of Large Numbers, given $\epsilon>0$, $\delta>0$, we
  can choose $L\equiv L_{\epsilon,\delta}$ such that
  \begin{equation}\label{eq:lb1}
    \Pr
    \left(S_n > -L - n(a+\epsilon)
    \right)
    \ge 1-\delta,
    \qquad
    n=0,1,2,\ldots.
  \end{equation}
  Then, for $x\ge0$,
  \begin{align*}
    \Pr(M>x)
    & = \sum_{n\ge0} \Pr(M_n\le x,\, S_{n+1}>x) \nonumber\\
    & \ge \sum_{n\ge0} \Pr(M_n\le x, \,
    S_n > -L - n(a+\epsilon), \,
    \xi_{n+1}>x+L+n(a+\epsilon)) \nonumber\\
    & = \sum_{n\ge0} \Pr(M_n\le x, \,
    S_n > -L - n(a+\epsilon))\,
    \Pr(\xi_{n+1}>x+L+n(a+\epsilon)) \nonumber\\
    & \ge \sum_{n\ge0} (1-\delta-\Pr(M_n>x))
    \overline{F}(x+L+n(a+\epsilon) \nonumber)\\
    & \ge (1-\delta-\Pr(M>x))
    \frac{\overline{F^s}(x+L)}{a+\epsilon},
  \end{align*}
  where the last line above follows since, from (\ref{eq:3}), for any
  $x$ and any constant~$b>0$,
  $\overline{F^s}(x)\le{}b\sum_{n\ge0}\overline{F}(x+nb)$.
  Rearranging, we obtain
  \begin{equation}
    \label{eq:lb}
    \Pr(M>x) \ge \frac{(1-\delta)\overline{F^s}(x+L)}%
    {a+\epsilon+\overline{F^s}(x+L)}
  \end{equation}
  Since $F^s$ is assumed long-tailed, it now follows that
  \begin{displaymath}
    \liminf_{x\to\infty}\frac{\Pr(M>x)}{\overline{F^s}(x)}
    \ge \frac{1-\delta}{a+\epsilon}.
  \end{displaymath}
  The result~(\ref{eq:5}) is thus obtained on letting $\delta,\epsilon\to0$.
  
  We now prove (ii).  Since $F^s$ is here assumed subexponential, it
  is in particular long-tailed.  Hence, by (i), it is sufficient to
  establish the upper bound associated with the result~(\ref{eq:6}).
  For a sequence of events $\{A_n\}$ we make the convention:
  $\min\{n\ge1:\I(A_n)=1\}=\infty$ if $\I(A_n)=0$ for all $n$.  Given
  $\epsilon\in(0,a)$ and $R>0$, define renewal times
  $0\equiv\tau_0<\tau_1\le\tau_2\le\dots$ for the process~$\{S_n\}$ by
  \begin{displaymath}\label{eq:7}
    \tau_1 = \min\{n\ge1:S_n>R-n(a-\epsilon)\} \le \infty,
  \end{displaymath}
  and, for $m\ge2$,
  \begin{align*}
    \tau_m & = \infty, \quad\text{if $\tau_{m-1}=\infty$},\\
    \tau_m & = \tau_{m-1}
    + \min\{n\ge1:S_{\tau_{m-1}+n}-S_{\tau_{m-1}}>R-n(a-\epsilon)\},
    \quad\text{if $\tau_{m-1}<\infty$}.
  \end{align*}
  Observe that
  \begin{equation}\label{eq:8}
    \{(\tau_m-\tau_{m-1},\,S_{\tau_m}-S_{\tau_{m-1}})\}_{n\ge1}
    \quad\text{ are i.i.d.}
  \end{equation}
  Then, again from (\ref{eq:4}), by the Strong Law of Large Numbers,
  \begin{equation}\label{eq:9}
    \gamma\equiv\Pr(\tau_1<\infty)\to0 \qquad\text{as $R\to\infty$}.
  \end{equation}
  Define also $S_\infty=-\infty$.  We now have that, for all
  sufficiently large $x$,
  \begin{align}
    \Pr(S_{\tau_1}>x) & = \sum_{n\ge1} \Pr(\tau_1=n,\,S_n>x) \nonumber\\
    & \le \sum_{n\ge1} \Pr(S_{n-1}\le
    R-(n-1)(a-\epsilon),\,S_n>x) \nonumber\\ & \le \sum_{n\ge1}
    \overline{F}(x-R+(n-1)(a-\epsilon)) \nonumber\\ & \le
    \frac{1}{a-\epsilon}\overline{F^s}(x-R-a+\epsilon).
    \label{eq:ub2}
  \end{align}
  where the last line above follows since, from (\ref{eq:3}), for all
  sufficiently large $x$ and any constant~$b>0$,
  $\overline{F^s}(x)\ge{}b\sum_{n\ge1}\overline{F}(x+nb)$.

  Let $\{\phi_m\}_{m\ge1}$ be independent identically distributed
  random variables such that
  \begin{displaymath}
    \Pr(\phi_1>x) = \Pr(S_{\tau_1}>x \,|\, \tau_1<\infty), \qquad
    x\in\mathbb{R}.
  \end{displaymath}
  Then, from \eqref{eq:ub2} and since $F^s$ is long-tailed,
  \begin{equation}
    \label{eq:10}
    \Pr(\phi_1>x) \le \overline{G}(x), \qquad  x\in\mathbb{R},
  \end{equation}
  for some distribution function~$G$ on $\R$ satisfying
   \begin{equation}
    \label{eq:11}
    \lim_{x\to\infty}\frac{\overline{G}(x)}{\overline{F^s}(x)}
    =\frac{1}{\gamma(a-\epsilon)}.
  \end{equation}
  We now have, for all $x>R-a+\epsilon$,
  \begin{align}
    \Pr(M>x) & \le \sum_{m\ge1}
    \Pr(M>x,\,S_{\tau_m}>x-R+a-\epsilon) \label{eq:15}\\
    & \le \sum_{m\ge1}
    \gamma^m\Pr(\phi_1+\dots+\phi_m>x-R+a-\epsilon) \label{eq:12}\\
    & \le \sum_{m\ge1} \gamma^m\overline{G^{*m}}(x-R+a-\epsilon), \label{eq:ub}
  \end{align}
  where (\ref{eq:12}) and (\ref{eq:ub}) follow from (\ref{eq:8}) and
  \eqref{eq:10} respectively.
  
  It follows from \eqref{eq:11}, the subexponentiality of $F^s$, and
  the well-known properties of subexponential distributions (see, for
  example, Embrechts \textit{et al.}\ (1997), Sigman (1999), or
  Asmussen (2000)) that $G$ is subexponential.  Hence from the
  well-known upper bound for convolutions of subexponential
  distributions (again see, for example, Embrechts \textit{et al.}\ 
  (1997, Lemma 1.3.5)), given any $k>1$, there exists $A>0$ such that
  \begin{equation}\label{eq:13}
    \frac{\overline{G^{*m}}(x)}{\overline{F^s}(x)}\le Ak^m
    \quad\text{for all $x$ and for all $m\ge0$,}
  \end{equation}
  and also, again from \eqref{eq:11} and by the subexponentiality of
  $G$, for all $m\ge1$,
  \begin{equation}
    \lim_{x\to\infty}\label{eq:14}
    \frac{\overline{G^{*m}}(x)}{\overline{F^s}(x)} =
    \frac{m}{\gamma(a-\epsilon)}.
  \end{equation}
  Hence by \eqref{eq:ub}, \eqref{eq:14}, the long-tailedness of $F^s$,
  and the dominated convergence theorem (justified by \eqref{eq:13} for
  any $k$ such that $k\gamma<1$),
  \begin{align*}
    \limsup_{x\to\infty}\frac{\Pr(M>x)}{\overline{F^s}(x)} & \le
    \frac{1}{a-\epsilon}\sum_{m\ge1}m\gamma^{m-1}\\
    &  = \frac{1}{(a-\epsilon)(1-\gamma)^2}.
  \end{align*}
  Now let $R\to\infty$, so that $\gamma\to0$ by \eqref{eq:9}, and
  then let $\epsilon\to0$ to obtain the required upper bound
  \begin{displaymath}
    \limsup_{x\to\infty}\frac{\Pr(M>x)}{\overline{F^s}(x)} \le \frac{1}{a},
  \end{displaymath}
  and hence the result~(\ref{eq:6}).
\end{proof}

\begin{rem}\label{rem:1}
  It follows from (\ref{eq:15}) and the subsequent argument that,
  given $\epsilon$, $R$, $\tau_1$, $\gamma$ as defined in the above
  proof of the upper bound,
  \begin{align}
    \lim_{x\to\infty}\frac{\Pr(M>x)}{\overline{F^s}(x)}
    & \le \limsup_{x\to\infty}
    \frac{\Pr(M>x,\,S_{\tau_1}>x-R+a-\epsilon)}{\overline{F^s}(x)}
    +\gamma\left(\frac{2-\gamma}{(a-\epsilon)(1-\gamma)^2}\right)
    \nonumber \\
    & =
    \limsup_{x\to\infty}\frac{\Pr(M>x,\,S_{\tau_1}>x)}{\overline{F^s}(x)}
    +\gamma\left(\frac{2-\gamma}{(a-\epsilon)(1-\gamma)^2}\right)
    \label{eq:16}\\
    & = (1+o(1))
    \limsup_{x\to\infty}\frac{\Pr(M>x,\,S_{\tau_1}>x)}{\overline{F^s}(x)}
    \quad\text{as $R\to\infty$.}
    \label{eq:17}
  \end{align}
  Here (\ref{eq:16}) follows from the observation that, for any $b>0$,
  $\Pr(M>x,\,S_{\tau_1}>x)\le\Pr(M>x,\,S_{\tau_1}>x-b)%
  \le\Pr(M>x-b,\,S_{\tau_1}>x-b)$, together with the long-tailedness
  of $F^s$ again, while (\ref{eq:17}) follows from (\ref{eq:9}).
  Thus, given \emph{any} $\epsilon\in(0,a)$, if $R$ is chosen
  sufficiently large, then as $x\to\infty$ the only significant way in
  which $M$ can exceed $x$ is that it should do so at the first
  time~$n$ such that $S_n$ exceeds $R-n(a-\epsilon)$.  This is the
  principle of ``one big jump''.
\end{rem}

\begin{rem}\label{rem:2}
  The proof of the lower bound above may be incorporated into that of
  the upper bound.  It is necessary to modify the
  definition of $\tau_1$ above to
  \begin{displaymath}
    \tau_1 = \min\{n\ge1:S_n<-R-n(a+\epsilon)
    \quad\text{or}\quad
    S_n > R-n(a-\epsilon)\}
  \end{displaymath}
  and to similarly modify the definition of the subsequent renewal
  times $\tau_2,\tau_3,\dots$ (see also Foss and Zachary (2002)).  The
  argument then proceeds in a two-sided version of that given above to
  obtain the final result more directly.  As a proof of Veraverbeke's
  Theorem this is somewhat tidier.  However, the separate proof of the
  lower bound as given here (which, as previously remarked, is part of
  the folklore of the theorem) makes it clear that only the condition
  that $F^s$ be long-tailed, rather than that it be subexponential, is
  required for that.  Further, the proof of the lower bound requires
  only the Weak, rather than the Strong, Law of Large Numbers.  Thus
  for given $\epsilon$, $\delta$, it is easy to determine a suitable
  constant~$L$ such that (\ref{eq:lb1}) holds, and so (\ref{eq:lb})
  may be used to give an explicit lower bound for $P(M>x)$ for any
  given $x$.  (Of course this should be optimised over $\epsilon$,
  $\delta$.)
  
  While the relation~(\ref{eq:ub}) in principle gives an upper bound
  for $P(M>x)$ for any given $x$ (in which, for small $\gamma$, the
  importance of accurately bounding convolutions is greatly reduced),
  the more explicit determination of useful upper bounds remains a
  challenging problem.  For further results on bounds in particular
  cases, including those of the Weibull and regularly varying
  distributions, see Kalashnikov and Tsitsiashvili (1999).
\end{rem}

Finally, we remark also that a proof of the converse of Veraverbeke's
Theorem---that (for $\E\xi_1$ finite and negative) the
relation~\eqref{eq:6} implies the subexponentiality of $F^s$---is given
by Korshunov (1997).

\section*{Acknowledgement}

The author is most grateful to Serguei Foss and to Takis
Konstantopoulos for some very helpful discussions, and to the referee
for some most helpful comments.

\end{document}